\documentclass{gtart_a}
\pdfoutput=1

\title{Classification of continuously transitive circle groups}

\author{James Giblin}
\givenname{James}
\surname{Giblin}
\address{Mathematics Institute\\University of Warwick\\\newline
Coventry, CV4 7AL\\UK}
\email{giblin@maths.warwick.ac.uk}
\urladdr{http://www.maths.warwick.ac.uk/~giblin/}

\author{Vladimir Markovic}
\givenname{Vladimir}
\surname{Markovic}
\email{markovic@maths.warwick.ac.uk}
\urladdr{http://www.maths.warwick.ac.uk/~markovic/}

\volumenumber{10}
\issuenumber{}
\publicationyear{2006}
\papernumber{32}
\lognumber{0685}
\startpage{1319}
\endpage{1346}

\doi{}
\MR{}
\Zbl{}


\received{12 December 2005}
\revised{22 June 2006}
\accepted{29 July 2006}
\published{18 September 2006}
\publishedonline{18 September 2006}
\proposed{David Gabai}
\seconded{Leonid Polterovich, Benson Farb}
\corresponding{}
\editor{}
\version{}

\arxivreference{}   

\AtBeginDocument{\let\bar\wbar%
\def\notin{\not\in}}
\makeop{id}

\makeatletter
\def\cnewtheorem#1[#2]#3{\newtheorem{#1}{#3}[section]
\expandafter\let\csname c@#1\endcsname\c@thm}

\newtheorem{thm}{Theorem}[section]
\cnewtheorem{prop}[thm]{Proposition}
\cnewtheorem{cor}[thm]{Corollary}
\cnewtheorem{lem}[thm]{Lemma}
\theoremstyle{definition}
\cnewtheorem{defn}[thm]{Definition}
\cnewtheorem{rmk}[thm]{Remark}
\cnewtheorem{exa}[thm]{Example}
\makeatother  

\newcommand{\bb}{\mathbb}
\renewcommand{\S}{$\bb{S}^{1}$}
\renewcommand{\H}{$\mathrm{Homeo}(\bb{S}^{1})$}

\begin{document}

\begin{asciiabstract}
Let G be a closed transitive subgroup of Homeo(S^1) which contains a
non-constant continuous path f: [0,1] --> G.  We show that up to
conjugation G is one of the following groups: SO(2,R), PSL(2,R),
PSL_k(2,R), Homeo_k(S^1), Homeo(S^1).  This verifies the
classification suggested by Ghys [Enseign. Math. 47 (2001) 329-407].
As a corollary we show that the group PSL(2,R) is a maximal closed
subgroup of Homeo(S^1) (we understand this is a conjecture of de la
Harpe).  We also show that if such a group G < Homeo(S^1) acts
continuously transitively on k-tuples of points, k>3, then the closure
of G is Homeo(S^1) (cf Bestvina's collection of `Questions in geometric
group theory').

\end{asciiabstract}

\begin{htmlabstract}
Let G be a closed transitive subgroup of
Homeo(<b>S</b><sup>1</sup>) which contains a non-constant
continuous path f:[0,1]&rarr;G.  We show that up to
conjugation G is one of the following groups:
SO(2,<b>R</b>), PSL(2,<b>R</b>),
PSL<sub>k</sub>(2,<b>R</b>), Homeo<sub>k</sub>(<b>S</b><sup>1</sup>),
Homeo(<b>S</b><sup>1</sup>).  This verifies the classification
suggested by Ghys in [Enseign. Math. 47 (2001) 329-407].  As a
corollary we show that the group PSL(2,<b>R</b>) is a maximal
closed subgroup of Homeo(<b>S</b><sup>1</sup>) (we understand this
is a conjecture of de la Harpe).  We also show that if such a group
G&lt;Homeo(<b>S</b><sup>1</sup>) acts continuously transitively on
k&ndash;tuples of points, k&gt;3, then the closure of G is
Homeo(<b>S</b><sup>1</sup>) (cf Bestvina's collection of `Questions
in geometric group theory').
\end{htmlabstract}

\begin{webabstract}
Let $G$ be a closed transitive subgroup of
$\mathrm{Homeo}(\mathbb{S}^1)$ which contains a non-constant
continuous path $f\colon\thinspace [0,1]\to G$.  We show that up to
conjugation $G$ is one of the following groups:
$\mathrm{SO}(2,\mathbb{R})$, $\mathrm{PSL}(2,\mathbb{R})$,
$\mathrm{PSL}_{k}(2,\mathbb{R})$, $\mathrm{Homeo}_{k}(\mathbb{S}^{1})$,
$\mathrm{Homeo}(\mathbb{S}^1)$.  This verifies the classification
suggested by Ghys in [Enseign. Math. 47 (2001) 329-407].  As a
corollary we show that the group $\mathrm{PSL}(2,\mathbb{R})$ is a maximal
closed subgroup of $\mathrm{Homeo}(\mathbb{S}^1)$ (we understand this
is a conjecture of de la Harpe).  We also show that if such a group
$G<\mathrm{Homeo}(\mathbb{S}^1)$ acts continuously transitively on
$k$--tuples of points, $k>3$, then the closure of $G$ is
$\mathrm{Homeo}(\mathbb{S}^1)$ (cf Bestvina's collection of `Questions
in geometric group theory').
\end{webabstract}

\begin{abstract}
Let $G$ be a closed transitive subgroup of
$\mathrm{Homeo}(\mathbb{S}^1)$ which contains a non-constant
continuous path $f\colon\thinspace [0,1]\to G$.  We show that up to
conjugation $G$ is one of the following groups:
$\mathrm{SO}(2,\mathbb{R})$, $\mathrm{PSL}(2,\mathbb{R})$,
$\mathrm{PSL}_{k}(2,\bb{R})$, $\mathrm{Homeo}_{k}(\bb{S}^{1})$,
$\mathrm{Homeo}(\mathbb{S}^1)$.  This verifies the classification
suggested by Ghys in \cite{GH}.  As a corollary we show that the group
$\mathrm{PSL}(2,\bb{R})$ is a maximal closed subgroup of
$\mathrm{Homeo}(\mathbb{S}^1)$ (we understand this is a conjecture of
de la Harpe).  We also show that if such a group
$G<\mathrm{Homeo}(\mathbb{S}^1)$ acts continuously transitively on
$k$--tuples of points, $k>3$, then the closure of $G$ is
$\mathrm{Homeo}(\mathbb{S}^1)$ (cf \cite{BE}).
\end{abstract}

\maketitle

\section{Introduction}
Let \H denote the group of orientation preserving homeomorphisms
of $\bb{S}^{1}$ which we endow with the uniform topology. Let $G$
be a subgroup of \H with the topology induced from \H. We say that $G$ is transitive if
for every two points $x,y \in \bb{S}^{1}$, there exists a map $f\in G$, such that $f(x)=y$.
We say that a group $G$ is closed if it is closed in the topology of \H.
A continuous path in $G$ is a continuous map $f\co[0,1]\to G$.

Let $\mathrm{SO}(2,\bb{R})$ denote the group of rotations of
$\bb{S}^{1}$ and $\mathrm{PSL}(2,\bb{R})$ the group of M\"obius
transformations.
The first main result we prove describes transitive subgroups of \H that contain a non constant continuous path.

\begin{thm}\label{thm:one}
Let $G$ be a transitive subgroup of \H which contains a non
constant continuous path. Then one of the following mutually
exclusive possibilities holds:
\begin{enumerate}
\item $G$ is conjugate to $\mathrm{SO}(2,\bb{R})$ in \H.
\item $G$ is conjugate to $\mathrm{PSL}(2,\bb{R})$ in \H.
\item For every $f\in$ \H and each finite set of points $x_{1},\dots,x_{n}\in\bb{S}^{1}$ there exists $g\in G$ such that $g(x_{i})=f(x_{i})$ for each $i$.
\item $G$ is a cyclic cover of a conjugate of $\mathrm{PSL}(2,\bb{R})$ in \H and hence conjugate to $\mathrm{PSL}_{k}(2,\bb{R})$ for some
$k>1$.
\item $G$ is a cyclic cover of a group satisfying condition 3 above.
\end{enumerate}
\end{thm}

Here we write $\mathrm{PSL}_{k}(2,\bb{R})$ and $\mathrm{Homeo}_{k}(\bb{S}^{1})$ to denote the cyclic covers of the groups
$\mathrm{PSL}(2,\bb{R})$ and $\mathrm{Homeo}(\bb{S}^{1})$ respectively, for some $k \in \bb{N}$.

The proof begins by showing that the assumptions of the theorem imply that $G$ is continuously 1--transitive. This means that if we vary points $x,y\in\bb{S}^{1}$ in a continuous fashion, then we can choose corresponding elements of $G$ which map $x$ to $y$ that also vary in a continuous fashion. In Theorems \ref{thm:so} and \ref{thm:nempty} we show that this leads us to two possibilities, either $G$ is conjugate to $\mathrm{SO}(2,\bb{R})$, or $G$ is a cyclic cover of a group which is continuously $2$--transitive.

We then analyse groups which are continuously 2--transitive and show that they are infact all continuously 3--transitive. Furthermore, if such a group is not continuously 4--transitive, we show that it is a convergence group and hence conjugate to $\mathrm{PSL}(2,\bb{R})$. On the other hand if it is continuously 4--transitive, then we use an induction argument to show that it is continuously $n$--transitive for all $n\geq 4$. This implies that for every $f\in$ \H and each finite set of points $x_{1},\dots,x_{n}\in\bb{S}^{1}$ there exists a group element $g$ such that $g(x_{i})=f(x_{i})$ for each $i$.

The remaining possibilities, namely cases 2 and 3, arise when the aforementioned cyclic cover is trivial.

In the case where the group $G$ is also closed we can use \fullref{thm:one} to make the following classification.

\begin{thm}
Let $G$ be a closed transitive subgroup of \H which contains a non
constant continuous path. Then one of the following mutually
exclusive possibilities holds:
\begin{enumerate}
\item{$G$ is conjugate to $\mathrm{SO}(2,\bb{R})$ in \H}.
\item{$G$ is conjugate to $\mathrm{PSL}_{k}(2,\bb{R})$ in \H for some $k\geq 1$}.
\item{$G$ is conjugate to $\mathrm{Homeo}_{k}(\bb{S}^{1})$ in \H for some $k\geq 1$}.
\end{enumerate}
\end{thm}

The above theorem provides the classification of closed, transitive subgroups of \H that contain a non-trivial continuous path. This classification was suggested by Ghys for all transitive and closed subgroups of \H (See \cite{GH}).

One well known problem in the theory of circle groups is to prove that the group of M\"obius transformations is a maximal
closed subgroup of \H. We understand that this is a conjecture of  de la Harpe (see \cite{BE}). The following theorem follows directly from our work and answers this question.

\begin{thm}
$\mathrm{PSL}(2,\bb{R})$ is a maximal closed subgroup of \H.
\end{thm}

In the following five sections we develop the techniques needed to prove our results. Here we prove the results about the transitivity
on $k$--tuples of points. In Section 7 we give the proofs of all the main results stated above.

\section{Continuous Transitivity}

Let $G<\mathrm{Homeo}(\bb{S}^{1})$ be a transitive group of
orientation preserving homeomorphisms of $\bb{S}^{1}$. We begin
with some definitions which generalize the notion of transitivity.

Set,
$$
P_{n}=\{(x_{1},\dots,x_{n}):x_{i}\in\bb{S}^{1},x_{i}=x_{j}\iff
i=j\}
$$
to be the set of distinct $n$--tuples of points in \S. Two $n$--tuples
$$
(x_{1},\dots,x_{n}),(y_{1},\dots,y_{n})\in P_{n}
$$
have matching orientations if there exists $f\in$ \H such that
$f(x_{i})=y_{i}$ for each $i$.

\begin{defn}
$G$ is $n$--transitive if for every pair
$(x_{1},\dots,x_{n}),(y_{1},\dots,y_{n})\in P_{n}$ with matching
orientations there exists $g\in G$ such that $g(x_{i})=y_{i}$ for
each $i$.
\end{defn}

\begin{defn}
$G$ is uniquely $n$--transitive if it is $n$--transitive
and for each pair $(x_{1},\dots,x_{n}),(y_{1},\dots,y_{n})\in
P_{n}$ with matching orientations there is exactly one element
$g\in G$ such that $g(x_{i})=y_{i}$. Equivalently, the only
element of $G$ fixing $n$ distinct points is the identity.
\end{defn}

Endow $\bb{S}^{1}$ with the standard topology and $P_{n}$ with the topology it
inherits as a subspace of the $n$--fold Cartesian product
$\bb{S}^{1}\times\cdots\times\bb{S}^{1}$. These are metric
topologies. With the topology on $P_{n}$ being induced by the distance function
$$
d_{P_{n}}((x_{1},\dots,x_{n}),(y_{1},\dots,y_{n}))=\max\{d_{\bb{S}^{1}}(x_{i},y_{i}):i=1,\dots,n\},
$$
where $d_{\bb{S}^{1}}$ is the standard Euclidean distance function on $\bb{S}^{1}$.

Endow $G$ with the uniform topology. This is also a metric topology, induced by the
distance function,
$$
d_{G}(g_{1},g_{2})=\sup\{\max\{d_{\bb{S}^{1}}(g_{1}(x),g_{2}(x)),d_{\bb{S}^{1}}(g_{1}^{-1}(x),g_{2}^{-1}(x))\}:x\in\bb{S}^{1}\}
$$

A path in a topological space $X$ is a continuous map
$\gamma\co[0,1]\to X$. If $\mathcal{X}\co[0,1]\to P_{n}$ is a path in
$P_{n}$ we will write $x_{i}(t)=\pi_{i}\circ \mathcal{X}(t)$,
where $\pi_{i}$ is projection onto the $i$--th component of
$\bb{S}^{1}\times\cdots\times\bb{S}^{1}$, so that we can write
$\mathcal{X}(t)=(x_{1}(t),\cdots,x_{n}(t))$. We will call a pair
of paths $\mathcal{X},\mathcal{Y}\co[0,1]\to P_{n}$
\emph{compatible} if there exists a path
$h\co[0,1]\to\mathrm{Homeo}(\bb{S}^{1})$ with
$h(t)(x_{i}(t))=y_{i}(t)$ for each $i$ and $t$.

\begin{defn}
$G$ is continuously $n$--transitive if for every
compatible pair of paths $\mathcal{X},\mathcal{Y}\co[0,1]\to P_{n}$
there exists a path $g\co[0,1]\to G$ with the property that
$g(t)(x_{i}(t))=y_{i}(t)$ for each $i$ and $t$.
\end{defn}

\begin{defn}
A continuous deformation of the identity in $G$ is a non
constant path of homeomorphisms $f_{t}\in G$ for $t\in[0,1]$ with
$f_{0}=\id$.
\end{defn}

We have the following lemma.

\begin{lem}\label{lem:equ}
For $n\geq 2$ the following are equivalent:
\begin{enumerate}
\item{$G$ is continuously $n$--transitive}.
\item{$G$ is continuously $n-1$--transitive and the following holds. 
For every $n-1$--tuple $(a_{1},\dots,a_{n-1})\in P_{n-1}$ and $x\in\bb{S}^{1}\setminus\{a_{1},\dots,a_{n-1}\}$ there exists
a continuous map $F_{x}\co I_{x}\to G$ satisfying the following
conditions,
\begin{enumerate}
\item{$F_{x}(y)$ fixes $a_{1},\dots,a_{n-1}$ for all $y\in I_{x}$}
\item{$(F_{x}(y))(x)=y$ for all $y\in I_{x}$}
\item{$F_{x}(x)=\id$}
\end{enumerate}\vspace{-5pt}
where $I_{x}$ is the component of $\bb{S}^{1}\setminus\{a_{1},\dots,a_{n-1}\}$ containing $x$.}
\item{$G$ is continuously $n-1$--transitive and there exists $(a_{1},\dots,a_{n-1})\in P_{n-1}$ with the following property.
There is a component $I$ of
$\bb{S}^{1}\setminus\{a_{1},\dots,a_{n-1}\}$, a point $\tilde{x}\in I$ and
a continuous map $F_{\tilde{x}}\co I\to G$ satisfying the following
conditions,
\begin{enumerate}
\item{$F_{\tilde{x}}(y)$ fixes $a_{1},\dots,a_{n-1}$ for all $y\in I$}
\item{$(F_{\tilde{x}}(y))(\tilde{x})=y$ for all $y\in I$}
\item{$F_{\tilde{x}}(\tilde{x})=\id$.}
\end{enumerate}}
\item{$G$ is continuously $n-1$--transitive and there exists $(a_{1},\dots,a_{n-1})\in P_{n-1}$ with the following property. There is a component $I$ of $\bb{S}^{1}\setminus\{a_{1},\dots,a_{n-1}\}$, such that for each $x\in I$ there exists a continuous deformation of the identity $f_{t}$, satisfying $f_{t}(a_{i})=a_{i}$ for each $t$ and $i$ and $f_{t}(x)\neq x$ for some $t$.}
\end{enumerate}
\end{lem}

\begin{proof}
We start by showing $[1\Rightarrow 4]$. As $G$ is continuously $n$--transitive, it will automatically be continuously $n-1$ transitive. Take $(a_{1},\dots,a_{n-1})\in P_{n-1}$ and $x\in\bb{S}^{1}\setminus\{a_{1},\dots,a_{n-1}\}$. Let $I_{x}$ be the component of $\bb{S}^{1}\setminus\{a_{1},\dots,a_{n-1}\}$ which contains $x$. Take $y\in I_{x}\setminus\{x\}$ and let $x_{t}$ be an injective path in $I_{x}$ with $x_{0}=x$ and $x_{1}=y$.

Let $\mathcal{X}\co [0,1]\to P_{n}$ be the constant path defined by $\mathcal{X}(t)=(a_{1},\dots,a_{n-1},x_{0})$ and let $\mathcal{Y}\co [0,1]\to P_{n}$ be the path defined by $\mathcal{Y}(t)=(a_{1},\dots,a_{n-1},x_{t})$. Then since $x_{t}\in I_{x}$ for every time $t$ these form an compatible pair of paths. Consequently, there exists a path $g_{t}\in G$ which fixes each $a_{i}$ and such that $g_{t}(x)=(x_{t})$. Defining $f_{t}=g_{t}\circ (g_{0}^{-1})$ gives us the required continuous deformation of the identity.

We now show that $[4\Rightarrow 3]$. For $\tilde{x}\in I$
 set $K_{\tilde{x}}$ to be the set of points $x\in I$ for which there is a path of homeomorphisms $f_{t}\in G$ satisfying,
\begin{enumerate}
\item{$f_{0}=\id$}
\item{$f_{t}(a_{i})=a_{i}$ for each $i$ and $t$}
\item{$f_{1}(\tilde{x})=x$.}
\end{enumerate}
Obviously, $K_{\tilde{x}}$ will be a connected subset of $I$ and
hence an interval for each $\tilde{x}\in I$.

Choose $\tilde{x}\in I$ and take $x\in K_{\tilde{x}}$. Let $f_{t}$ and $g_{t}$ be
continuous deformations of the identity which
fix the $a_{i}$ for all $t$ and such that
$f_{t_{0}}(x)\neq x$ for some $t_{0}\in(0,1]$ and
$g_{1}(\tilde{x})=x$. $f_{t}$ exists by the assumptions of condition 4. and $g_{t}$ exists because $x\in K_{\tilde{x}}$. The following paths show that the
interval between $f_{t_{0}}(x)$ and $(f_{t_{0}})^{-1}(x)$ is
contained in $K_{\tilde{x}}$:
$$
h_{1}(t)=\left\{\begin{array}{ll} g_{2t} & t\in[0,1/2]\\
f_{t_{0}(2t-1)}\circ g_{1} & t\in[1/2,1] \end{array}\right.
$$
$$
h_{2}(t)=\left\{\begin{array}{ll} g_{2t} & t\in[0,1/2]\\
(f_{t_{0}(2t-1)})^{-1}\circ g_{1} & t\in[1/2,1]
\end{array}\right. 
$$
As $x$ is contained in this interval and cannot be equal to either
of its endpoints we see that $K_{\tilde{x}}$ is open for every
$\tilde{x}\in I$. On the other hand, $\tilde{x}\in K_{\tilde{x}}$
for each $\tilde{x}\in I$ and if $x_{1}\in K_{x_{2}}$ then
$K_{x_{1}}=K_{x_{2}}$. Consequently, the sets
$\{K_{\tilde{x}}:\tilde{x}\in I\}$ form a partition of $I$ and
hence $K_{\tilde{x}}=I$ for every $\tilde{x}\in I$.

We now construct the map $F_{\tilde{x}}$. To do this, take a
nested sequence of intervals $[x_{n},y_{n}]$ containing
$\tilde{x}$ for each $n$ and such that $x_{n},y_{n}$ converge to
the endpoints of $I$ as $n\to\infty$. We define $F_{\tilde{x}}$
inductively on these intervals. Since $K_{\tilde{x}}=I$ we can
find a path of homeomorphisms $f_{t}\in G$ satisfying,
\begin{enumerate}
\item{$f_{0}=\id$}
\item{$f_{t}(a_{i})=a_{i}$ for each $i$ and $t$}
\item{$f_{1}(\tilde{x})=x_{1}$.}
\end{enumerate}
We now show that there exists a path $\bar{f}_{t}\in G$, which
also satisfies the above, but with the additional condition that
the path $\bar{f}_{t}(\tilde{x})$ is simple.

To see this, let $[x^{*},\tilde{x}]$ be the largest subinterval of
$[x_{1},\tilde{x}]$ for which there exists a path
$\bar{f}_{t}\in G$ which satisfies,
\begin{enumerate}
\item{$\bar{f}_{0}=\id$}
\item{$\bar{f}_{t}(a_{i})=a_{i}$ for each $i$ and $t$}
\item{$\bar{f}_{1}(\tilde{x})=x^{*}$}
\item{$\bar{f}_{t}(\tilde{x})$ is simple.}
\end{enumerate}
We want to show that $x^{*}=x_{1}$. Assume for contradiction that
$x^{*}\neq x_{1}$. Then since $x^{*}\in[x_{1},\tilde{x}]$ there
exists $s\in[0,1]$ such that $f_{s}(\tilde{x})=x^{*}$ and for
small $\epsilon>0$, we have that
$f_{s+\epsilon}(\tilde{x})\notin[x^{*},\tilde{x}]$. Then if we
concatenate the path $\bar{f}_{t}$ with $f_{s+\epsilon}\circ
f_{s}^{-1}\circ\bar{f}_{1}$ for small $\epsilon$ we can construct
a simple path satisfying the same conditions as $\bar{f}_{t}$ but
on a interval strictly bigger than $[x^{*},\tilde{x}]$, this
contradicts the maximality of $x^{*}$ and we deduce that
$x^{*}=x_{1}$.

We can use the path $\bar{f}_{t}$ to define a map
$F_{\tilde{x}}^{1}\co [x_{1},y_{1}]\to G$ satisfying,

\begin{enumerate}
\item{$F_{\tilde{x}}^{1}(y)$ fixes each $a_{i}$ for each $y\in I$}
\item{$(F_{\tilde{x}}^{1}(y))(\tilde{x})=y$ for all $y\in I$}
\item{$F_{\tilde{x}}^{1}(\tilde{x})=\id$.}
\end{enumerate}

by taking paths of homeomorphisms that move $\tilde{x}$ to $x_{1}$
and $y_{1}$ along simple paths in $\bb{S}^{1}$.

Now assume we have defined a map
$F_{\tilde{x}}^{k}\co [x_{k},y_{k}]\to G$ satisfying,

\begin{enumerate}
\item{$F_{\tilde{x}}^{k}(y)$ fixes each $a_{i}$ for each $y\in I$}
\item{$(F_{\tilde{x}}^{k}(y))(\tilde{x})=y$ for all $y\in I$}
\item{$F_{\tilde{x}}^{k}(\tilde{x})=\id$.}
\end{enumerate}

We can use the same argument used to produce $F_{\tilde{x}}^{1}$
to show that there exists a map
$\mathfrak{F}_{x_{k}}\co [x_{k+1},x_{k}]\to G$ such that
$\mathfrak{F}_{x_{k}}(x)$ fixes the $a_{i}$ for each $x$,
$\mathfrak{F}_{x_{k}}(x_{k})=\id$ and
$(\mathfrak{F}_{x_{k}}(x))(x_{k})=x$. Similarly there exists a map
$\mathfrak{F}_{y_{k}}\co [y_{k},y_{k+1}]\to G$ such that
$\mathfrak{F}_{y_{k}}(x)$ fixes the $a_{i}$ for each $x$,
$\mathfrak{F}_{y_{k}}(y_{k})=\id$ and
$(\mathfrak{F}_{y_{k}}(x))(y_{k})=x$.

This allows us to define, $F_{\tilde{x}}^{k+1}\co [x_{k+1},y_{k+1}]\to G$ by:
$$
F_{\tilde{x}}^{k+1}(x)=\left\{\begin{array}{ll} F_{\tilde{x}}^{k}(x) & x\in[x_{k},y_{k}]\\
(\mathfrak{F}_{x_{k}}(x))\circ F_{\tilde{x}}^{k}(x_{k}) & x\in[x_{k+1},x_{k}]\\
(\mathfrak{F}_{y_{k}}(x))\circ F_{\tilde{x}}^{k}(y_{k}) & x\in[y_{k},y_{k+1}] \end{array}\right.
$$

Inductively, we can now define the full map $F_{\tilde{x}}\co I\to G$.

We now show that $[3\Rightarrow 2]$. So take $x'\in I$ with $x'\neq \tilde{x}$ and define $F_{x'}\co I\to G$ by
\begin{equation}\label{eqn:fx}
F_{x'}(y)=F_{\tilde{x}}(y)\circ(F_{\tilde{x}}(x'))^{-1}
\end{equation}
Then $F_{x'}$ satisfies,
\begin{enumerate}
\item{$F_{x'}(y)$ fixes $a_{1},\dots,a_{n-1}$ for all $y\in I$}
\item{$(F_{x'}(y))(x')=y$ for all $y\in I$}
\item{$F_{x'}(x')=\id$.}
\end{enumerate}
 Moreover, we can use \eqref{eqn:fx} to define a map $F\co I\times I\to G$ which is continuous in each variable and satisfies,
\begin{enumerate}
\item{$F(x,y)$ fixes $a_{1},\dots,a_{n-1}$ for all $x,y\in I$}
\item{$(F(x,y))(x)=y$ for all $x,y\in I$}
\item{$F(x,x)=\id$ for all $x\in I$.}
\end{enumerate}

Now take
$x'$ to be a point in $\bb{S}^{1}\setminus I\cup\{a_{1},\dots,a_{n-1}\}$ and let
$I'$ be the component of
$\bb{S}^{1}\setminus\{a_{1},\dots,a_{n-1}\}$ which contains $x'$.
Then since $G$ is continuously $n-1$--transitive there exists $g\in
G$ which permutes the $a_{i}$ so that $g(I)=I'$. Define
$F_{x'}\co I'\to G$ by
$$
F_{x'}(y)=g\circ F_{g^{-1}(x')}(g^{-1}(y)) \circ g^{-1}
$$
for $y\in I'$. Then $F_{x'}$ satisfies,
\begin{enumerate}
\item{$F_{x'}(y)$ fixes $a_{1},\dots,a_{n-1}$ for all $y\in I$}
\item{$(F_{x'}(y))(x')=y$ for all $y\in I'$}
\item{$F_{x'}(x')=\id$.}
\end{enumerate}

Now let $(b_{1},\dots,b_{n-1})\in P_{n-1}$ have the same
orientation as $(a_{1},\dots,a_{n-1})$ then since $G$ is
continuously $n-1$--transitive there exists $g\in G$ so that
$g(a_{i})=b_{i}$ for each $i$. Let
$x'\in\bb{S}^{1}\setminus\{b_{1},\dots,b_{n-1}\}$ and let $I'$ be
the component of $\bb{S}^{1}\setminus\{b_{1},\dots,b_{n-1}\}$ in
which it lies. Define $F_{x'}\co I'\to G$ by
$$
F_{x'}(y)=g\circ F_{g^{-1}(x')}(g^{-1}(y)) \circ g^{-1}
$$
for $y\in I'$. Then $F_{x'}$ satisfies,
\begin{enumerate}
\item{$F_{x'}(y)$ fixes $b_{1},\dots,b_{n-1}$ for all $y\in I$}
\item{$(F_{x'}(y))(x')=y$ for all $y\in I'$}
\item{$F_{x'}(x')=\id$}
\end{enumerate}
 and we have that $[3\Rightarrow 2]$

Finally we have to show that $[2\Rightarrow 1]$. Let $\mathcal{X},\mathcal{Y}\co [0,1]\to P_{n}$ be an compatible pair of paths. We define $\mathcal{X'}\co [0,1]\to P_{n-1}$ by
$$
\mathcal{X'}(t)=(x_{1}(t),\dots, x_{n-1}(t))
$$
and $\mathcal{Y'}\co [0,1]\to P_{n-1}$ by
$$
\mathcal{Y}'(t)=(y_{1}(t),\dots, y_{n-1}(t)).
$$
Notice that $\mathcal{X}'$ and $\mathcal{Y}'$ will also be a compatible pair of paths. Furthermore, as $G$ is continuously $n-1$--transitive there will exist a path $g'\co [0,1]\to G$ such that $g'(t)(x_{i}(t))=y_{i}(t)$ for $1\leq i\leq n-1$.

The paths $\mathcal{X}',\mathcal{Y}'\co [0,1]\to P_{n-1}$ will also be compatible with the constant paths,
$$\mathcal{X}'_{0}\co [0,1]\to P_{n-1}$$
$$\mathcal{X}'_{0}(t)=\mathcal{X}'(0)$$
$$\mathcal{Y}'_{0}\co [0,1]\to P_{n-1}\leqno{\rm and}$$
$$\mathcal{Y}'_{0}(t)=\mathcal{Y}'(0)$$
respectively. So that there exist paths $g_{x}',g_{y}'\co [0,1]\to G$ with $g_{x}'(x_{i}(0))=x_{i}(t)$ and $g_{y}'(y_{i}(0))=y_{i}(t)$ for $1\leq i\leq n-1$. Furthermore, by pre composing with $(g_{x}'(0))^{-1}$ and $(g_{y}'(0))^{-1}$ if necessary, we can assume that $g_{x}'(0)=g_{y}'(0)=\id$.

We now construct a path $g_{x}\co [0,1]\to G$ which satisfies,
$$
g_{x}(t)(x_{i}(0))=x_{i}(t)
$$
for $1\leq i\leq n$. To do this let
$I$ be the component of
$\bb{S}^{1}\setminus\{x_{1}(0),\dots,x_{n-1}(0)\}$ containing
$x_{n}(0)$. By assumption we have a continuous map
$F_{x_{n}(0)}\co I\to G$ satisfying
\begin{enumerate}
\item{$F_{x_{n}(0)}(y)$ fixes $x_{1}(0),\dots,x_{n-1}(0)$ for all $y\in I$}
\item{$(F_{x_{n}(0)}(y))(x)=y$ for all $y\in I$}
\item{$F_{x_{n}(0)}(x)=\id$.}
\end{enumerate}
Define $g_{x}\co [0,1]\to G$ by
$$
g_{x}(t)=g_{x}'(t)\circ(F_{x_{n}(0)}((g_{x}'(t))^{-1}(x_{n}(t))))^{-1}.
$$
Then $g_{x}(t)(x_{i}(0))=x_{i}(t)$ for $1\leq i\leq n$. We can repeat this process with $g_{y}'$ to construct a path $g_{y}\co [0,1]\to G$ satisfying $g_{y}(t)(y_{i}(0))=y_{i}(t)$ for $1\leq i\leq n$.

The map $g'(0)$ which we defined earlier will map $x_{i}(0)$ to $y_{i}(0)$ for $1\leq i\leq n-1$. Moreover, $g'(0)(x_{n}(0))$ will lie in the same component of $\bb{S}^{1}\setminus\{y_{1}(0),\dots y_{n-1}(0)\}$ as $y_{n}(0)$. So we have a map $F_{g'(0)(x_{n}(0))}(y_{n}(0))$ which maps $g'(0)(x_{n}(0))$ to $y_{n}(0)$ and fixes the other $y_{i}(0)$. Putting all of this together allows us to define $g\co [0,1]\to G$ by
$$
g(t)=g_{y}(t)\circ F_{g'(0)(x_{n}(0))}(y_{n}(0))\circ g'(0)\circ(g_{x}(t))^{-1}.
$$
This is a path in $G$ which satisfies $g_{t}(x_{i}(t))=y_{i}(t)$ for each $i$ and $t$. Since we can do this for any two compatible paths, $G$ is continuously $n$--transitive and we have shown that $[2\Rightarrow 1]$.
\end{proof}

\begin{prop}\label{prop:one}
If $G$ is $1$--transitive and there exists a continuous deformation
of the identity $f_{t}\co [0,1]\to G$ in $G$, then $G$ is
continuously $1$--transitive.
\end{prop}

\begin{proof}
Let $x_{0}\in\bb{S}^{1}$ be such that $f_{t_{0}}(x_{0})\neq x_{0}$ for some $t_{0}\in[0,1]$. Take $x\in \bb{S}^{1}$ then there exists $g\in G$ such that $g(x)=x_{0}$. Consequently, $g^{-1}\circ f_{t}\circ g$ is a continuous deformation of the identity which doesn't fix $x$ for some $t$. Since these deformations exist for each $x\in\bb{S}^{1}$ the proof follows in exactly the same way as $[4\Rightarrow 1]$ from the proof of \fullref{lem:equ}.
\end{proof}

From now on we will assume that $G$ contains a continuous deformation of the
identity, and hence is continuously 1--transitive.

\section{The set $J_{x}$}

\begin{defn}
For $x\in\bb{S}^{1}$ we define $J_{x}$ to be the set of points
$y\in\bb{S}^{1}$ which satisfy the following condition. There exists a continuous deformation of
the identity $f_{t}\in G$ which fixes $x$ for all $t$ and such
that $f_{t_{0}}(y)\neq y$ for some $t_{0}\in[0,1]$.
\end{defn}

It follows directly from this definition that $x\notin J_{x}$.

\begin{lem}
$J_{f(x)}=f(J_{x})$ for every $f\in G$ and $x\in\bb{S}^{1}$.
\end{lem}

\begin{proof}
Let $y\in J_{f(x)}$ and let $f_{t}$ be the corresponding
continuous deformation of the identity with $f_{t_{0}}(y)\neq y$.
Then $f^{-1}\circ f_{t}\circ f$ is also a continuous deformation
of the identity which now fixes $x$, and for which
$f_{t_{0}}(f^{-1}(y))\neq f^{-1}(y)$. This means that $f^{-1}(y)\in
J_{x}$ and hence $y\in f(J_{x})$ so that
$J_{f(x)}\subset f(J_{x})$. The other inclusion is an
identical argument.
\end{proof}

\begin{lem}
$J_{x}$ is open for every $x\in\bb{S}^{1}$.
\end{lem}

\begin{proof}
Let $y\in J_{x}$ and take $f_{t}$ to be the corresponding
continuous deformation of the identity with $f_{t_{0}}(y)\neq y$
for some $t_{0}\in[0,1]$. Then since $f_{t_{0}}$ is continuous there exists a neighborhood $U$ of $y$ such that
$f_{t_{0}}(z)\neq z$ for all $z\in U$. This implies that $U\subset
J_{x}$ and hence that $J_{x}$ is open.
\end{proof}

\begin{lem}\label{lem:fin}
$J_{x}=\emptyset$ for every $x\in\bb{S}^{1}$ or $J_{x}$ has a finite complement for every $x\in\bb{S}^{1}$.
\end{lem}

To prove this lemma we will use the Hausdorff maximality Theorem which we now recall.

\begin{defn}
A set $\mathcal{P}$ is partially ordered by a binary relation $\leq$ if,
\begin{enumerate}
\item{$a\leq b$ and $b\leq c$ implies $a\leq c$}
\item{$a\leq a$ for every $a\in \mathcal{P}$}
\item{$a\leq b$ and $b\leq a$ implies that $a=b$.}
\end{enumerate}
\end{defn}

\begin{defn}
A subset $\mathcal{Q}$ of a partially ordered set $\mathcal{P}$ is
totally ordered if for every pair $a,b\in\mathcal{Q}$ either
$a\leq b$ or $b\leq a$. A totally ordered subset
$\mathcal{Q}\subset\mathcal{P}$ is maximal if for any member
$a\in\mathcal{P}\setminus\mathcal{Q}$, $\mathcal{Q}\cup\{a\}$ is
not totally ordered.
\end{defn}

\begin{thm}[Hausdorff Maximality Theorem]\label{thm:hau}
Every nonempty partially ordered set contains a maximal totally ordered subset.
\end{thm}

We now prove \fullref{lem:fin}.

\begin{proof}
Assume that there exists $x\in\bb{S}^{1}$ for which $J_{x}=\emptyset$. Then for every $y\in\bb{S}^{1}$ there exists a map $g\in G$ such that $g(x)=y$. Consequently,
$$
J_{y}=J_{g(x)}=g(J_{x})=g(\emptyset)=\emptyset
$$
for every $y\in\bb{S}^{1}$.

Assume that $J_{x}\neq\emptyset$ for every $x\in\bb{S}^{1}$ and let $S_{x}=\bb{S}^{1}\setminus J_{x}$ denote the complement of $J_{x}$. This means that $S_{x}$ consists of the points $y\in\bb{S}^{1}$ such every continuous deformation of the identity which fixes $x$ also fixes $y$. The set $\mathcal{P}=\{S_{x}:x\in\bb{S}^{1}\}$ is partially ordered by inclusion so that by \fullref{thm:hau} there exists a maximal totally ordered subset, $\mathcal{Q}=\{S_{x}:x\in A\}$, where $A$ is the appropriate subset of $\bb{S}^{1}$.

If we set $\mathcal{S}=\bigcap_{x\in A}S_{x}$ then we have the following:
\begin{enumerate}
\item{$\mathcal{S}\neq\emptyset$}\label{it1}
\item{if $x\in\mathcal{S}$ then $S_{x}=\mathcal{S}$.}\label{it2}
\end{enumerate}
\ref{it1} follows from the fact that $\mathcal{S}$ is the intersection of
a descending family of compact sets, and hence is nonempty.

To see that \ref{it2} is also true, fix $x\in\mathcal{S}$. Then from the definition of $\mathcal{S}$, we will have $x\in S_{a}$ for each $a\in A$. In other words, if we take $a\in A$, then every continuous deformation of the identity which fixes $a$ will also fix $x$. Furthermore, if $y\in S_{x}$ then every continuous deformation of the identity which fixes $a$ not only fixes $x$ but $y$ too, so that $S_{x}\subset S_{a}$. This is true for every $a\in A$ so that $S_{x}\subset\mathcal{S}$. On the other hand, by the maximality of $\mathcal{Q}$, it must contain $S_{x}$. Consequently, if $x\in\mathcal{S}$ then $S_{x}=\mathcal{S}$.

Fix $x_{0}\in\mathcal{S}$ and assume for contradiction that $S_{x_{0}}$ is infinite. Take a sequence $x_{n}\in S_{x_{0}}$ and let $x_{n_{k}}$ be a convergent subsequence with limit $x'$. This limit will also be in $S_{x_{0}}$ as it is closed. As $J_{x_{0}}$ is a nonempty open subset of $\bb{S}^{1}$ it will contain an interval $(a,b)$ with $a,b\in S_{x_{0}}$. Take maps $g_{a},g_{b}\in G$ so that $g_{a}(x')=a$ and $g_{b}(x')=b$. Since $x',a \in S_{x_{0}}$ we have that,
$$
g_{a}(S_{x_{0}})=g_{a}(S_{x'})=S_{g_{a}(x')}=S_{a}=S_{x_{0}}
$$
and similarly for $g_{b}$. As a result $g_{a}(x_{n}),g_{b}(x_{n})\in S_{x_{0}}$ for each $n$, but $g_{a},g_{b}$ are orientation preserving homeomorphisms so that at least one of these points will lie in $(a,b)$, a contradiction.

We have shown that $S_{x_{0}}$ is finite. If we now take any other point $x\in\bb{S}^{1}$ then there exists a map $g\in G$ such that $g(x_{0})=x$. This means that the set $S_{x}=S_{g(x_{0})}=g(S_{x_{0}})$ will also be finite and we are done.
\end{proof}

\begin{thm}\label{thm:so}
If $J_{x}=\emptyset$ for all $x\in\bb{S}^{1}$ then $G$ is
conjugate in $\mathrm{Homeo}(\bb{S}^{1})$ to the group of
rotations $\mathrm{SO}(2,\bb{R})$.
\end{thm}

We require the following lemma for the proof of this Theorem.

\begin{lem}\label{lem:aff}
If $f\co \bb{R}\to\bb{R}$ is a homeomorphism which conjugates translations to translations, then it is an affine map.
\end{lem}

\begin{proof}
Let $f$ be a homeomorphism which conjugates translations to
translations and set $f_{1}=T\circ f$ where $T$ is the translation
that sends $f(0)$ to 0. Then $f_{1}$ fixes 0 and also conjugates
translations to translations. In particular there exists $\alpha$
such that $f_{1}$ conjugates $x\mapsto x+1$ to the map $x\mapsto
x+\alpha$. Notice that $\alpha\neq 0$ since the identity is only
conjugate to itself.

Now define $f_{2}=f_{1}\circ M_{\alpha}$ where
$M_{\alpha}(x)=\alpha x$.  A simple calculation shows that $f_{2}$
conjugates $x\mapsto x+1$ to itself and conjugates translations to
translations. Since $f_{2}$ fixes 0 and conjugates $x\mapsto x+1$
to itself, we deduce that it must fix all the integer points.

Now, for $n\in\bb{N}$ let $\gamma\in\bb{R}$ be such that
$(f_{2})^{-1}\circ T_{1/n}\circ f_{2}=T_{\gamma}$ where
$T_{\alpha}(x)=x+\alpha$. It follows that,
$$
T_{1}=(f_{2})^{-1}\circ (T_{1/n})^n \circ f_{2}=((f_{2})^{-1}\circ
T_{1/n}\circ f_{2})^{n}=(T_{\gamma})^{n}
$$
so that $\gamma=1/n$ and $(f_{2})^{-1}\circ T_{1/n}\circ
f_{2}=T_{1/n}$ for every $n\in\bb{N}$. Combining this with the
fact that $f_{2}$ fixes 0, we see that $f_{2}$ must fix all the
rational points and hence is the identity. This implies that
$f_{1}$ and hence $f$ are affine.
\end{proof}

We can now prove \fullref{thm:so}.

\begin{proof}
Let $\widehat{G}<G$ denote the path component of the identity in
$G$. We are going to show that $\widehat{G}$ is a compact group.
Proposition 4.1 in \cite{GH} will then imply that it is conjugate
in $\mathrm{Homeo}(\bb{S}^{1})$ to a subgroup of
$\mathrm{SO}(2,\mathbb{R})$. Moreover, as $\widehat{G}$ is
1--transitive it will be equal to the whole of
$\mathrm{SO}(2,\mathbb{R})$.

For $x\in\bb{S}^{1}$ let $\pi_{x}\co \bb{R}\to\bb{S}^{1}$ be
the usual projection map which sends each integer to $x$ and for each integer translation
$T\co \bb{R}\to\bb{R}$
satisfies $\pi_{x} \circ T=\pi_x$.

If we fix $x\in\bb{S}^{1}$ then since $G$ is continuously 1--transitive we can
choose a continuous path $g\co [0,1]\to G$ such that $g(t)(x)=\pi_{x}(t)$
and $g(0)=\id$. Notice that this path is contained in $\widehat{G}$
and $g(1)$ is not necessarily the identity even though it fixes
$x$.

For $x\in\bb{S}^{1}$ we define a continuous map $F_{x}\co \bb{R}\to \widehat{G}$ by
$$
F_{x}(t)=g(t-[t])\circ g(1)^{[t]}
\eqno(*)
$$
where $[t]$ is the greatest integer less than or equal to $t$. Set $f=F_x(1)$.
Note that $F_x(n)=f^{n}$ for every $n \in \bb{Z}$.

We claim that $F_{x}$ has the following properties,

\begin{enumerate}
\item $F_{x}(t)(x)=\pi_{x}(t)$ for every $t \in \bb{R}$
\item$F_{x}(0)=\id$
\item The map $F_{x}$ is a surjection, that is $F_{x}(\bb{R})=\widehat{G}$
\item If the map $f=F_x(1)$ is not equal to the identity map then $F_x$ is a bijection
\end{enumerate}

The first two properties follow directly from the definition. To see that the third property holds, let $h_s$ be a path in
$\widehat{G}$, $s \ge 0$, $h_0=\id$. Let $\alpha(s)=h_s(x)$. We have that $\alpha$ is a continuous map from the non-negative reals $\bb{R}^{+}$
into the circle. Since the set $\bb{R}^{+}$ is contractible we can lift the map $\alpha$ into the universal cover of the circle.
That is, there is a map $\beta\co \bb{R}^{+} \to \bb{R}$ such that $\pi_{x} \circ \beta=\alpha$. We have $F_x(\beta(s))(x)=h_s(x)$.
Then $({h_s}^{-1} \circ F_x(\beta(s)))(x)=x$. It follows from the assumption of the theorem that $F_x(\beta(s))=h_s$ and $F_x$ is surjective.
The map $F_x$ is injective for $0 \le t <1$, because $F_{x}(t)(x)=\pi_{x}(t)$.
If $F_x(1)$ is not the identity, and since  $F_x(1)(x)=x$ we have that $F_x(m)=F_x(n)$ if and only if  $m=n$, for every two integers $m,n$.
This implies the fourth property.

It follows from $(*)$, and the surjectivity of $F_x$, that $\widehat{G}$ is a compact group if and only if the cyclic group generated by
$F_x(1)=f$ is a compact group. We will prove that $f=\id$.

Assume that $f$ is not the identity map. Since $F_{x}$ is a bijection for each $t \in\bb{R}$ there exists a unique $s_n(t)\in\bb{R}$ such that,
$$
f^{n}\circ F_{x}(t) \circ f^{-n}=F_{x}(s_n(t)).
\eqno(**)
$$
This defines a function $s_n\co \bb{R}\to\bb{R}$ which we claim is continuous for each $n$. To see this, fix $n$ and let $t_{m}\in\bb{R}$ be a convergent sequence with limit $t'$. Since $F_{x}$ is continuous,
$$
f^{n}\circ F_{x}(t_m) \circ f^{-n}\longrightarrow f^{n}\circ F_{x}(t') \circ f^{-n}
$$
and so $F_{x}(s_n(t_m))\to F_{x}(s_n(t'))$ as $m\to\infty$.

Now, if $s_n(t_{m_k})$ is a convergent subsequence, with limit $t_0$, then using continuity $F_{x}(s_n(t_{m_k}))$ will converge to $F_{x}(t_0)$. Since $F_{x}$ is a bijection this gives us that $t_{0}=s_{n}(t')$. Consequently, if the sequence $s_{n}(t_{m})$ were bounded, then it would converge to $t'$.

Assume now that the sequence $s_{n}(t_{m})$ is unbounded and take a divergent subsequence $s_{n}(t_{m_{k}})$. Consider the corresponding sequence,
$$
F_{x}(s_n(t_{m_{k}}))=g(s_n(t_{m_{k}})-[s_n(t_{m_{k}})])\circ f^{[s_n(t_{m_{k}})]}.
$$
Since $s_n(t_{m_{k}})-[s_n(t_{m_{k}})]\in [0,1)$ for each $m$, there exists a subsequence $t_{m_{k_{l}}}$ of $t_{m_{k}}$ such that $s_n(t_{m_{k_{l}}})-[s_n(t_{m_{k_{l}}})]$ converges to some $t_{0}\in[0,1]$. Now since $g$ is continuous and the sequence $F_{x}(s_n(t_{m}))$ converges to a homeomorphism $F_{x}(s_{n}(t'))$ we have that $f^{[s_{n}(t_{m_{k_{l}}})]}$ converges to a homeomorphism as $l\to\infty$. However, as $s_{n}(t_{m_{k}})$ is divergent $[s_{n}(t_{m_{k_{l}}})]$ will be divergent too.

Let $S_{f}$ denote the set of fixed points of $f$. Note that $x \in S_f$. Since we assume that $f$ is not the identity we have that
$\bb{S}^{1} \setminus S_f$ is non-empty. Let $J$ be a component of $\bb{S}^{1} \setminus S_f$ and let $a,b \in \bb{S}^{1}$ be its endpoints. Since $f$ fixes $J$, and has no fixed points inside $J$ we deduce that on compact subsets of $J$ the sequence $f^{[s_{n}(t_{m_{k_{l}}})]}$ converges to one of the endpoints and consequently, can not converge to a homeomorphism. This is a contradiction, so $s_{n}(t_{m})$ can not be unbounded and $s_{n}$ is continuous.

Notice that $s_n(0)=0$ and if $t \in \bb{Z}$ then $F_{x}(t)$ will commute with
the $f^{n}$ so we have  $s_n(m)=m$ for all $m \in \bb{Z}$. This yields that $s_n([0,1])=[0,1]$ for every $n \in \bb{Z}$.

Let $U_f \subset \bb{S}^{1}$ be the set defined as follows. We say that $y \in U_f$ if there exists an open interval $I$, $y \in I$,
such that $|f^{n}(I)| \to 0$, $n \to \infty$. Here $|f^{n}(I)|$ denotes the length of the corresponding interval. The set $U_f$ is open.
We show that $U_f$ is non-empty and not equal to $\bb{S}^{1}$. As before, let $J$ be a component of $\bb{S}^{1} \setminus S_f$ and let $a,b \in \bb{S}^{1}$ be its endpoints. Since $f$ fixes $J$, and has no fixed points inside $J$ we deduce that on compact subsets of $J$ the sequence $f^n$ converges to one of the endpoints, say $a$. This shows that $J \subset U_f$. Also, this shows that the point $b$ does not belong to $U_f$.

Let $y \in U_f$, and let $I$ be the corresponding open interval so that  $y \in I$ and $|f^{n}(I)| \to 0$, $n \to \infty$.
Set $f^{n}(I)=I_n$. Consider the interval $F_{x}(s_n(t))(I_n)$, $t \in [0,1]$. Since $s_n([0,1])=[0,1]$ we have that $F_{x}(s_n([0,1]))$ is a compact family of homeomorphisms. This allows us to conclude that
$|F_{x}(s_n(t))(I_n)| \to 0$, $n \to \infty$, uniformly in $n$ and $t \in [0,1]$.
Set $J_t=F_x(t)(I)$. From $(**)$ we have that $|f^{n}(J_t)| \to 0$, $n \to \infty$, for a fixed $t \in [0,1]$. This implies that
the point $F_x(t)(y)$ belongs to the set $U_f$ for every $t \in [0,1]$.

Let $J$ be a component of  $U_f$, and let $a,b$ be its endpoints. Note that the points $a,b$ do not belong to $U_f$.
Since $F_x(t)$ is a continuous path and $F_x(0)=\id$, for small enough $t$ we have that $F_x(t)(J) \cap J \ne \emptyset$. Since
$F_x(t)(J) \subset U_f$, and since $a,b$ are not in $U_f$ we have that  $F_x(t)(J)=J$. By continuity this extends to hold for every
$t \in [0,1]$. But this means that $F_x(t)(a)=a$ for every $t \in [0,1]$. However, for appropriately chosen inverse
$t_0=\pi_x^{-1}(a)$, we have that $F_x(t_0)(x)=a$, which contradicts the fact that $F_x(t_0)$ is a homeomorphism. This shows that
$f=\id$, and therefore we have proved that $\widehat{G}$ is a compact group.

To finish the argument, it remains to show that $G=\widehat{G}$. Let $\Phi\in\mathrm{Homeo}(\bb{S}^{1})$ be a
map which conjugates $\widehat{G}$ to $\mathrm{SO}(2,\mathbb{R})$ and take $g\in G\setminus\widehat{G}$.
Since $\widehat{G}$ is a normal subgroup of $G$, $\Phi\circ g\circ\Phi^{-1}$ conjugates rotations to rotations.
Lifting to the universal cover we get that every lift of $\Phi\circ g\circ\Phi^{-1}$ conjugates translations to translations.
If we choose one then by \fullref{lem:aff} it will be affine. On the other hand, it must be periodic, and hence is a translation.
So that $\Phi\circ g\circ\Phi^{-1}$ is itself a rotation and we are done.
\end{proof}

\begin{thm}\label{thm:nempty}
If $J_{x}\neq\emptyset$ then one of the following is true:
\begin{enumerate}
\item{$J_{x}=\bb{S}^{1}\setminus\{x\}$ in which case $G$ is continuously 2--transitive.}
\item{There exists $R\in\mathrm{Homeo}(\bb{S}^{1})$ which is conjugate to a finite order rotation and satisfies $R\circ g=g\circ R$ for every $g\in G$. Moreover, $G$ is a cyclic cover of a group $G_{\Gamma}$ which is continuously 2--transitive, where the covering transformations are the cyclic group generated by $R$.}
\end{enumerate}
\end{thm}

\begin{proof}
If $J_{x}=\bb{S}^{1}\setminus\{x\}$ then we are in case 4 of \fullref{lem:equ} with $n=2$. In this situation we know that $G$ will be continuously 2--transitive.

We already know that $S_{x}=\bb{S}^{1}\setminus J_{x}$ must contain $x$ and by \fullref{lem:fin} must be finite. Moreover, as $f(J_{x})=J_{f(x)}$ the sets $S_{x}$ contain the same number of points for each $x\in\bb{S}^{1}$. Define $R\co \bb{S}^{1}\to\bb{S}^{1}$ by taking $R(x)$ to be the first point of $S_{x}$ you come to as you travel anticlockwise around $\bb{S}^{1}$. Now take $g\in G$ and $x\in\bb{S}^{1}$, then since $J_{g(x)}=g(J_{x})$ and $g$ is orientation preserving $R\circ g(x)=g\circ R(x)$ for all $x\in\bb{S}^{1}$.

We now show that $R$ is a homeomorphism. To see this take any continuous path $x_{t}\in\bb{S}^{1}$, we will show that $R(x_{t})\to R(x_{0})$ as $t\to 0$. Since $G$ is continuously 1--transitive, there exists a continuous path $g_{t}\in G$ satisfying $g_{t}(x_{t})=x_{0}$, so that,
$$
\lim_{t\to 0}R(x_{t})=\lim_{t\to 0}(g_{t})^{-1}(R(g_{t}(x_{t})))=\lim_{t\to 0}(g_{t})^{-1}(R(x_{0}))=R(x_{0}).
$$
where the first equality follows from the fact that $R\circ g(x)=g\circ R(x)$ for all $x\in\bb{S}^{1}$. This shows that $R$ is continuous. If we take $y\notin J_{x}$ then $J_{x}\subset J_{y}$, and hence $S_{x}\supset S_{y}$ but in this case since $S_{x}$ and $S_{y}$ contain the same number of points they will be equal. Consequently, $R$ has an inverse defined by taking $R^{-1}(x)$ to be the first point of $S_{x}$ you come to by traveling clockwise around $\bb{S}^{1}$ and this inverse is continuous by the same argument as for $R$. Consequently, $R\in\mathrm{Homeo}(\bb{S}^{1})$. Furthermore, $R$ is of finite order equal to the number of points in $S_{x}$ and hence conjugate to a rotation.

Let $\Gamma$ denote the cyclic subgroup of \H generated by $R$. Define $\pi\co \bb{S}^{1}\to\bb{S}^{1}/\Gamma\cong\bb{S}^{1}$, in the usual way with $\pi(x)$ being the orbit of $x$ under $\Gamma$. Since $R\circ g(x)=g\circ R(x)$ for all $x\in\bb{S}^{1}$, each $g\in G$ defines a well defined homeomorphism of the quotient space $\bb{S}^{1}/\Gamma$ which we call $g_{\Gamma}$. This gives us a homomorphism $\pi_{\Gamma}\co G\to\mathrm{Homeo}(\bb{S}^{1})$, defined by $\pi_{\Gamma}(g)=g_{\Gamma}$. Let $G_{\Gamma}$ denote the image of $G$ under $\pi_{\Gamma}$, then $G$ is a cyclic cover of $G_{\Gamma}$.

It remains to see that $G_{\Gamma}$ is continuously 2--transitive. This follows from the fact that if we take $x_{0}\in\bb{S}^{1}$ then $J_{\pi(x_{0})}=\pi(J_{x_{0}})$, where $J_{\pi(x_{0})}$ is the set of points that can be moved by continuous deformations of the identity in $G_{\Gamma}$ which fix $\pi(x_{0})$. Consequently, $J_{\pi(x_{0})}=\bb{S}^{1}\setminus\{x_{0}\}$ so that $G_{\Gamma}$ is continuously 2--transitive by the first part of this proposition.
\end{proof}

\section{Implications of continuous 2--transitivity}

We now know that if $G$ is transitive and contains a continuous deformation of the identity then it is either conjugate to the group of rotations $\mathrm{SO}(2,\bb{R})$, is continuously 2--transitive, or is a cyclic cover of a group which is continuously 2--transitive. For the rest of the paper we assume that $G$ is continuously 2--transitive and examine which possibilities arise.

For $n\geq 2$ and $(x_{1}\dots x_{n})\in P_{n}$ we define
$J_{x_{1}\dots x_{n}}$ to be the subset of $\bb{S}^{1}$ containing
the points $x\in\bb{S}^{1}$ which satisfy the following condition.
There exists a continuous deformation of the identity $f_{t}\in
G$, with $f_{t}(x_{i})=x_{i}$ for each $i$ and $t$ and such that
there exists $t_{0}\in[0,1]$ with $f_{t_{0}}(x)\neq x$. This
generalizes the earlier definition of $J_{x}$ and we get the
following analogous results.

\begin{lem}
$J_{f(x_{1})\dots f(x_{n})}=f(J_{x_{1}\dots x_{n}})$ for every
$f\in G$.
\end{lem}

\begin{lem}\label{lem:op}
$J_{x_{1}\dots x_{n}}$ is open.
\end{lem}

We also have the following.

\begin{lem}\label{lem:jeq}
If $J_{x_{1}\dots x_{n}}$ is nonempty and $G$ is continuously
$n$--transitive, then it is equal to
$\bb{S}^{1}\setminus\{x_{1}\dots x_{n}\}$.
\end{lem}

\begin{proof}
Assume that $J_{x_{1}\dots
x_{n}}\subset\bb{S}^{1}\setminus\{x_{1},\dots,x_{n}\}$ is
nonempty. By \fullref{lem:op} it is also open and hence is a
countable union of open intervals. Pick one of these, and call its
endpoints $b_{1}$ and $b_{2}$. Assume for contradiction that at
least one of $b_{1}$ and $b_{2}$ is not one of the $x_{i}$.
Interchanging $b_{1}$ and $b_{2}$ if necessary we can assume that
this point is $b_{1}$. Since $G$ is continuously $n$--transitive
there exist elements of $G$ which cyclically permute the $x_{i}$.
Using these elements and the fact that $J_{f(x_{1})\dots
f(x_{n})}=f(J_{x_{1}\dots x_{n}})$ for every $f\in G$, we can
assume without loss of generality that $b_{1}$ and hence the whole
interval lies in the component of \S
$\setminus\{x_{1},\dots,x_{n}\}$ whose endpoints are $x_{1}$ and
$x_{2}$.

We now claim that $J_{b_{1},b_{2},x_{3},\dots,x_{n}}\supset
J_{x_{1}\dots x_{n}}$. To see this, take $x\in J_{x_{1}\dots
x_{n}}$, then there exists a continuous deformation of the
identity $f_{t}$ which fixes $x_{1},\dots,x_{n}$ and for which
there exists $t_{0}$ such that $f_{t_{0}}(x)\neq x$. Now since
$b_{1},b_{2}\notin J_{x_{1}\dots x_{n}}$, $f_{t}$ must also fix
$b_{1}$ and $b_{2}$ for all $t$, consequently we can use $f_{t}$
to show that $x\in J_{b_{1},b_{2},x_{3},\dots,x_{n}}$. In
particular, this means that $J_{b_{1},b_{2},x_{3},\dots,x_{n}}$
contains the whole interval between $b_{1}$ and $b_{2}$.

Take $g\in G$ which maps $\{b_{1},b_{2}\}$ to $\{x_{1},x_{2}\}$
and fixes the other $x_{i}$, such an element exists as $G$ is
continuously $n$--transitive. Then,
$$
J_{x_{1},x_{2},x_{3},\dots
,x_{n}}=J_{g(b_{1}),g(b_{2}),g(x_{3}),\dots
,g(x_{n})}=g(J_{b_{1},b_{2},x_{3},\dots,x_{n}})
$$
so that $J_{x_{1}\dots x_{n}}$ must contain the whole interval
between $x_{1}$ and $x_{2}$. This is a contradiction, since
$b_{1}$ lies between $x_{1}$ and $x_{2}$ but is not in
$J_{x_{1}\dots x_{n}}$.
\end{proof}

\begin{prop}\label{prop:ind}
Let $G$ be continuously $n$--transitive for some $n\geq 2$ and
suppose there exist $n$ distinct points
$a_{1},\dots,a_{n}\in\bb{S}^{1}$ and a continuous deformation of
the identity $g_{t}\in G$, which fixes each $a_{i}$ for all $t$.
Then $G$ is continuously $n+1$ transitive.
\end{prop}

\begin{proof}
$J_{a_{1}\dots a_{n}}\neq\emptyset$ so by \fullref{lem:jeq} $J_{a_{1}\dots a_{n}}=\bb{S}^{1}\setminus\{a_{1},\dots,a_{n}\}$. We can now apply \fullref{lem:equ} to see that $G$ is continuously $n+1$--transitive.
\end{proof}

\begin{cor}\label{cor:int}
If $G$ is continuously 2--transitive and there exists $g\in
G\setminus\{\id\}$ with an open interval $I\subset\bb{S}^{1}$ such
that the restriction of $g$ to $I$ is the identity, then $G$ is
continuously $n$--transitive for every $n\geq 2$.
\end{cor}

\begin{proof}
Let $I\subset\bb{S}^{1}$ be a maximal interval on which $g$ acts
as the identity, so that if $I'\supset I$ is another interval containing
$I$ then $g$ doesn't act as the identity on $I'$. Let $a$ and $b$
be the endpoints of $I$ and  let $a_{t}$ and $b_{t}$ be continuous
injective paths with $a_{0}=a$,$b_{0}=b$ and $a_{t},b_{t}\notin I$
for each $t\neq 0$. This is possible because $g\neq \id$ so
that $\bb{S}^{1}\setminus I$ will be a closed interval containing
more than one point. Let $g_{t}$ be a continuous path in $G$ so
that $g_{0}=\id$, $g_{t}(a)=a_{t}$ and $g_{t}(b)=b_{t}$, such a
path exists as $G$ is continuously 2--transitive.

Consider the path $h_{t}=g^{-1} \circ g_{t}\circ g\circ
g_{t}^{-1}$ since $g_{0}=\id$ we get $h_{0}=\id$. Now $g_{t}\circ
g\circ g_{t}^{-1}$ acts as the identity on the interval between
$a_{t}$ and $b_{t}$ and by maximality of $I$, $g^{-1}$ will not act as the identity for $t\neq 0$. Consequently, $h_{t}$ is a
continuous deformation of the identity which acts as the identity
on $I$. So if $G$ is continuously $k$--transitive for $k\geq 2$, by
taking $k$--points in $I$ and using \fullref{prop:ind} we
get that $G$ is $k+1$--transitive. As a result, since $G$ is
continuously 2--transitive it will be $n$--transitive for every
$n\geq 2$.
\end{proof}

$\mathrm{SO}(2,\bb{R})$ is an example of a subgroup
of Homeo($\bb{S}^{1}$) which is continuously 1--transitive
but not continuously 2--transitive. However, as the next result
shows, there are no subgroups of Homeo($\bb{S}^{1}$) which
are continuously $2$--transitive but not continuously
$3$--transitive.

\begin{prop}\label{prop:ttt}
If $G$ is continuously $2$--transitive, then it is continuously
$3$--transitive.
\end{prop}

\begin{proof}
Let $a,b\in\bb{S}^{1}$ be distinct points. Construct two injective
paths $a(t),b(t)$ in $\bb{S}^{1}$ with disjoint images, such that
$a(0)=a$, $b(0)=b$ and such that $a(t)$ and $b(t)$ lie in the same
component of $\bb{S}^{1}\setminus\{a,b\}$ for $t\in(0,1]$. We
label this component $I$ and the other $I'$.

Since $G$ is continuously 2--transitive, there exists a path
$g(t)\in G$ such that $g(0)=\id$, $g(t)(a)=a(t)$ and $g(t)(b)=b(t)$
for every $t$. Now for every $t$ the restriction of $g(t)$ to
the closure of $I$, is a continuous map of a closed interval into itself, and
hence must have a fixed point, $c(t)$. This point will normally
not be unique, but since $g(t)$ is continuous, for a small enough time interval we can choose
it to depend continuously on $t$. Likewise for the restriction of
$g(t)^{-1}$ to the closure of $I'$, for a small enough time interval we can
choose a path of fixed points $d(t)$, which must therefore also be
fixed points for $g(t)$.

Now pick points $c\in I$ and $d\in I'$. Using continuous
2--transitivity of $G$ construct a path $h(t)\in G$ such that
$h(t)(c)=c(t)$ and $h(t)(d)=d(t)$. Then $h_{t}^{-1}\circ
g(t)\circ h_{t}$ is only the identity when $t=0$ because the same is true of $g(t)$ and we have
constructed a continuous deformation of the identity which fixes
$c$ and $d$ for all $t$. Consequently we can use \fullref{prop:ind} to show that $G$ is continuously $3$--transitive.
\end{proof}

\section{Convergence Groups}\label{sec:con}

\begin{defn}
A subgroup $G$ of \H is a convergence group if for every sequence of
distinct elements $g_{n}\in G$, there exists a subsequence
$g_{n_{k}}$ satisfying one of the following two properties:
\begin{enumerate}
\item{There exists $g\in G$ such that,
$$
\lim_{k\to\infty}g_{n_{k}}=g\hspace{.3cm}{\rm and}\hspace{.3cm}\lim_{k\to\infty}g_{n_{k}}^{-1}=g^{-1}
$$
uniformly in $\bb{S}^{1}$.}
\item{There exist points $x_{0},y_{0}\in\bb{S}^{1}$ such that,
$$
\lim_{k\to\infty}g_{n_{k}}=x_{0}\hspace{.3cm}{\rm and}\hspace{.3cm}\lim_{k\to\infty}g_{n_{k}}^{-1}=y_{0}
$$
uniformly on compact subsets of $\bb{S}^{1}\setminus\{y_{0}\}$ and
$\bb{S}^{1}\setminus\{x_{0}\}$ respectively.}
\end{enumerate}
\end{defn}

The notion of convergence groups was introduced by Gehring and Martin \cite{GM} and they have proceeded to play a central role in geometric group theory. The following theorem has been one of the most important and we shall make frequent use of it.

\begin{thm}\label{thm:gab}
$G$ is a convergence group if and only if it is conjugate in
$\mathrm{Homeo}(\bb{S}^{1})$ to a subgroup of $\mathrm{PSL}(2,\bb{R})$.
\end{thm}

This Theorem was proved by Gabai in \cite{GA}. Prior to that, Tukia \cite{TU} proved this result in many cases and Hinkkanen \cite{HI} proved it for non discrete groups. Casson and Jungreis proved it independently using different methods \cite{CJ}. See \cite{CJ}, \cite{GA}, \cite{TU} for references to other papers in this subject.

For the rest of this section we shall assume that $G$ is continuously $n$--transitive, but not continuously $n+1$--transitive for some $n\geq 3$.

Take $(x_{1},\dots,x_{n-1})\in P_{n-1}$ and define
$$
G_{0}=\{g\in G:g(x_{i})=x_{i} \hspace{.3cm} i=1,\dots,n-1\}.
$$
Choose a component $I$ of $\bb{S}^{1}\setminus\{x_{1},\dots,x_{n-1}\}$ and denote its closure by $\bar{I}$. We construct a homomorphism $\Phi\co G_{0}\to\mathrm{Homeo}(\bb{S}^{1})$ as follows. Take $g\in G_{0}$, then since $g$ fixes the endpoints of $I$ and is orientation preserving, we can restrict it to a homeomorphism $g'$ of $\bar{I}$. By identifying the endpoints of $\bar{I}$ we get a copy of $\bb{S}^{1}$ and we define $\Phi(g)$ to be the homeomorphism of $\bb{S}^{1}$ that $g'$ descends to under this identification. We label the identification point $\bar{x}$ and set $\mathcal{G}_{0}=\Phi(G_{0})$ to be the image of $G_{0}$ under $\Phi$.

In this situation \fullref{lem:equ} implies the following. For every $x\in I$, there exists a continuous map $F_{x}\co 
\bb{S}^{1}\setminus\bar{x}\to \mathcal{G}_{0}$ satisfying the properties,

\begin{enumerate}
\item{$(F_{x}(y))(x)=y$\hspace{.3cm}$\forall$ $y\in \bb{S}^{1}\setminus\bar{x}$}
\item{$F_{x}(x)=\id$.}
\end{enumerate}

\begin{prop}\label{prop:iso}
$\Phi\co G_{0}\to\mathcal{G}_{0}$ is an isomorphism.
\end{prop}

\begin{proof}
Surjectivity is trivial. If we assume that $\Phi$ is not injective then there will exist $g\in G_{0}$ which is non-trivial and acts
as the identity on $I$. Then by \fullref{cor:int} $G$ will be $n+1$ transitive, a contradiction.
\end{proof}

Let $\widehat{G}_{0}$ denote the path component of the identity in
$G_{0}$, we now analyze the group $\widehat{\mathcal{G}}_{0}=\Phi(\widehat{G}_{0})$.

\begin{prop}\label{prop:red}
$\widehat{\mathcal{G}}_{0}$ is a
convergence group.
\end{prop}

\begin{proof}
Choose $x\in I$ then we know there exists a continuous map $F_{x}\co \bb{S}^{1}\setminus\bar{x}\to \mathcal{G}_{0}$ satisfying the properties,

\begin{enumerate}
\item{$(F_{x}(y))(x)=y$\hspace{.3cm}$\forall$ $y\in \bb{S}^{1}\setminus\bar{x}$}
\item{$F_{x}(x)=\id$.}
\end{enumerate}

Now since $F_{x}(x)=\id$ and
$F_{x}$ is continuous, the image of $F_{x}$ will lie entirely in
$\widehat{\mathcal{G}}_{0}$.

In fact, $F_{x}$ gives a bijection between $\bb{S}^{1}\setminus\bar{x}$ and $\widehat{\mathcal{G}}_{0}$.
To see this we first observe that injectivity follows directly from condition 1. To see that it is also surjective, take $g\in \widehat{\mathcal{G}}_{0}$. Then there
exists a path $g_{t}\in\widehat{\mathcal{G}}_{0}$ for $t\in[0,1]$ with $g_{0}=\id$ and
$g_{1}=g$. So that $g_{t}(x)$ is a path in $\bb{S}^{1}\setminus\bar{x}$ from $x$ to $g(x)$.
Consider the path $(F_{x}(g_{t}(x)))^{-1}\circ g_{t}$ in
$\widehat{\mathcal{G}}_{0}$, it fixes $x$ for every $t$, and so must be the
identity for each $t$. Otherwise, by \fullref{prop:ind},
$G$ would be continuously $n+1$--transitive, which would contradict
our assumptions. As a result $g=F_{x}(g(x))$ so $F_{x}$ is a
bijection, with inverse given by evaluation at $x$.

Fix $x_{0}\in \bb{S}^{1}\setminus\bar{x}$, let $g_{n}$ be a sequence of elements of
$\widehat{\mathcal{G}}_{0}$ and consider the sequence of points
$g_{n}(x_{0})$, since $\bb{S}^{1}$ is compact $g_{n}(x_{0})$ has a
convergent subsequence $g_{n_{k}}(x_{0})$ converging to some point
$x'$. If $x'\neq \bar{x}$ then by continuity of
$F_{x_{0}}$, $g_{n_{k}}$ will converge to $F_{x_{0}}(x')$. Now if
there does not exist a subsequence of $g_{n}(x_{0})$ converging to
some $x'\neq \bar{x}$, then take a subsequence $g_{n_{k}}$ such that
$g_{n_{k}}(x_{0})$ converges to $\bar{x}$. If we can show that
$g_{n_{k}}(x)$ converges to $\bar{x}$ for every $x\in \bb{S}^{1}\setminus\bar{x}$ then we
shall be done.

Suppose for contradiction that there exists $x\in \bb{S}^{1}\setminus\bar{x}$ such that
$g_{n_{k}}(x)$ does not converge to $\bar{x}$. Then there exists a
subsequence of $g_{n_{k}}(x)$ which converges to $x'\neq \bar{x}$,
but then by the previous argument the corresponding subsequence of
$g_{n_{k}}$ will converge to the homeomorphism $F_{x}(x')$. This
is a contradiction since $F_{x}(x')(x_{0})$ would have to equal
$\bar{x}$.
\end{proof}

\begin{cor}\label{cor:fix}
Let $g$ be an element of $\widehat{\mathcal{G}}_{0}$. If $g$ fixes a point in $\bb{S}^{1}\setminus\bar{x}$ then it is the identity.
\end{cor}

\begin{proof}
Let $x\in \bb{S}^{1}\setminus\bar{x}$ be a fixed point of $g$. From the previous proof we know that $F_{x}\co I\to\widehat{\mathcal{G}}_{0}$ is a bijection. So that $F_{x}(g(x))=g$, but $g$ fixes $x$ so that $g=F_{x}(x)=\id$.
\end{proof}

\begin{cor}\label{cor:tran}
The restriction of the action of $\widehat{\mathcal{G}}_{0}$ to $\bb{S}^{1}\setminus\bar{x}$ is
conjugate to the action of $\bb{R}$ on itself by translation.
\end{cor}

\begin{proof}
By \fullref{thm:gab} and \fullref{prop:red}
$\widehat{\mathcal{G}}_{0}$ is conjugate in \H to a subgroup of
$\mathrm{PSL}(2,\bb{R})$ which fixes the point $\bar{x}$.
Moreover, from \fullref{cor:fix} this is the only point
fixed by a non trivial element. By identifying $\bb{S}^{1}$ with
$\bb{R}\cup\{\infty\}$ so that $\bar{x}$ is identified with
$\{\infty\}$ in the usual way, we see that
$\widehat{\mathcal{G}}_{0}$ is conjugate to a subgroup of the
M\"obius group acting on $\bb{R}\cup\{\infty\}$. Since every
element will fix $\{\infty\}$, their restriction to $\bb{R}$ will
be an element of Aff($\bb{R}$) acting without fixed points, so can
only be a translation. On the other hand the group must act
transitively on $\bb{R}$ and so must be the full group of
translations. This gives the result.
\end{proof}

\begin{prop}
The restriction of the action of $G_{0}$ to $I$ is conjugate to
the action of a subgroup of the affine group $\mathrm{Aff}(\bb{R})$ on $\bb{R}$. In
particular, each non trivial element of $G_{0}$ can act on $I$
with at most one fixed point.
\end{prop}

\begin{proof}
The restriction of $\widehat{\mathcal{G}}_{0}$ to $\bb{S}^{1}\setminus\bar{x}$ is isomorphic to the restriction of $\widehat{G}_{0}$ to $I$. So that by \fullref{cor:tran} there exists a homeomorphism
$\phi\co I\to\bb{R} $ which conjugates the restriction of $\widehat{G}_{0}$ to $I$, to the action of $\bb{R}$ on itself
by translation. Take $h\in G_{0}\setminus\widehat{G}_{0}$ then
$h'=\phi\circ h\circ\phi^{-1}$ is a self-homeomorphism of
$\bb{R}$. Since $\widehat{G}_{0}$ is a normal subgroup of $G_{0}$,
$h'$ conjugates every translation to another one and so by \fullref{lem:aff} is itself an affine map and the proof is complete.
\end{proof}

Let $g$ be a
nontrivial element of $G_{0}$, then $g\in\widehat{G}_{0}$ if and only if it acts on each component of
$\bb{S}^{1}\setminus\{x_{1},\dots,x_{n-1}\}$ as a conjugate of a
non trivial translation. Furthermore, if $g\notin\widehat{G}_{0}$
then it acts on each component of
$\bb{S}^{1}\setminus\{x_{1},\dots,x_{n-1}\}$ as a conjugate of a
affine map which is not a translation, each of which must have a
fixed point. This situation cannot actually arise as the next
proposition will show.

\begin{prop}\label{prop:equ}
$G_{0}=\widehat{G}_{0}$
\end{prop}

\begin{proof}
Let $g\in G_{0}\setminus\widehat{G}_{0}$, then $g$ acts on each
component of $\bb{S}^{1}\setminus\{x_{1},\dots,x_{n-1}\}$ as a
conjugate of a affine map which is not a translation.
Consequently, $g$ will have a fixed point in each component of
$\bb{S}^{1}\setminus\{x_{1},\dots,x_{n-1}\}$. Label the fixed
points of $g$ in the components of
$\bb{S}^{1}\setminus\{x_{1},\dots,x_{n-1}\}$ whose boundaries both
contain $x_{1}$ as $y_{1}$ and $y_{2}$. Since $G$ is $n$--transitive,
there exists a map $g'$ which sends $y_{1}$ to $x_{1}$ and fixes
all the other $x_{i}$. Then $g'\circ g\circ (g')^{-1}$ fixes all
the $x_{i}$ and hence is an element of $G_{0}$. On the other hand,
$g'\circ g\circ (g')^{-1}$ also fixes $g'(x_{1})$ and $g'(y_{2})$
which lie in the same component of
$\bb{S}^{1}\setminus\{x_{1},\dots,x_{n-1}\}$, this is impossible
since every non-trivial element of $G_{0}$ can only have one fixed
point in each component of
$\bb{S}^{1}\setminus\{x_{1},\dots,x_{n-1}\}$.
\end{proof}

\begin{cor}\label{cor:trans}
The restriction of the action of $G_{0}$ to $I$ is conjugate to the action of $\bb{R}$ on itself by translation. In particular the action is free.
\end{cor}

We finish this section by comparing the directions that a
non-trivial element of $G_{0}$ moves points in different
components of $\bb{S}^{1}\setminus\{x_{1},\dots,x_{n-1}\}$. So
endow $\bb{S}^{1}$ with the anti-clockwise orientation, this gives
us an ordering on any interval $I\subset\bb{S}^{1}$, where for
distinct points $x,y\in I$, $x\prec y$ if one travels in an
anti-clockwise direction to get from $x$ to $y$ in $I$. Let $g\in
G_{0}\setminus\{\id\}$ if $I$ is a component of
$\bb{S}^{1}\setminus\{x_{1},\dots,x_{n-1}\}$ then we shall say
that $g$ acts \emph{positively} on $I$ if $x\prec g(x)$ and
\emph{negatively} if $x\succ g(x)$ for one and hence every $x\in
I$.

Let $I$ and $I'$ be the two components of
$\bb{S}^{1}\setminus\{x_{1},\dots,x_{n-1}\}$ whose boundaries
contain $x_{i}$. Labeled so that in the order on the closure of $I$, $x\prec
x_{i}$ for each $x\in I$, whereas in the order on the closure of $I'$,
$x_{i}\prec x$ for each $x\in I'$. Then we have the following,

\begin{prop} \label{prop:dir}

Let $g$ be a non trivial element of $G_{0}$, if $g$ acts
positively on $I$ then it acts negatively on $I'$ and if $g$ acts
negatively on $I$ then it acts positively on $I'$.
\end{prop}

\begin{proof}
Let $x,x'\in I$ and $y,y'\in I'$ be points such that $x\prec x'$
and $y\succ y'$. There exists $g\in G$ fixing
$x_{1},\dots,x_{i-1}$ and $x_{i+1},\dots,x_{n-1}$ and sending $x$
to $x'$ and $y$ to $y'$. This map will have a fixed point
$\tilde{x}$ between $x'$ and $y'$, since it maps the interval
between them into itself.

Let $g'\in G$ fix $x_{1},\dots,x_{i-1}$ and
$x_{i+1},\dots,x_{n-1}$ and send $\tilde{x}$ to $x_{i}$. Then
$g_{0}=g'\circ g\circ (g')^{-1}$ will fix $x_{1},\dots,x_{n-1}$
and hence lie in $G_{0}$. Moreover, $g_{0}$ acts positively on $I$
and negatively on $I'$.

Now let $g_{1}\in G_{0}$ be any non-trivial element which acts
positively on $I$. Then there exists a path $g_{t}$ in $G_{0}$
from $g_{0}=g'\circ g\circ (g')^{-1}$ to $g_{1}$, so that
$g_{t}\neq \id$ for any $t$. Since $g_{t}$ is never the identity
and $g_{0}$ acts negatively on $I'$, $g_{1}$ must also act
negatively on $I'$.

If $h\in G_{0}$ is a non-trivial element which acts negatively on
$I$, then $h^{-1}$ will act positively on $I$. So that, by the
above argument, $h^{-1}$ will act negatively on $I'$. This means
that $h$ will act positively on $I'$ as required.
\end{proof}

\begin{cor}
If $G$ is $n$--transitive but not $n+1$--transitive for $n\geq 3$
then $n$ is odd.
\end{cor}

\begin{proof}
Let $g$ be a non-trivial element of $G_{0}$ which acts positively
on some component $I$ of
$\bb{S}^{1}\setminus\{x_{1},\dots,x_{n-1}\}$. Then by \fullref{prop:dir} as we travel around $\bb{S}^{1}$ in an
anti-clockwise direction the manner in which it acts on each
component will alternate between negative and positive.
Consequently, if $n$ was even, when we return to $I$ we would
require that $g$ acted negatively on $I$, a contradiction, so $n$
is odd.
\end{proof}

\section{Continuous 3--transitivity and beyond}

We begin this section by analyzing the case where $G$ is
continuously 3--transitive but not continuously 4--transitive. We
shall show that such a group is a convergence group and
consequently conjugate to a subgroup of $\mathrm{PSL}(2,\bb{R})$.

Fix distinct points $x_{0},y_{0}\in\bb{S}^{1}$ and define
$$
G_{0}=\{g\in G:g(x_{0})=x_{0},g(y_{0})=y_{0}\}
$$
$$
\bar{G}=\{g\in G: g(x_{0})=x_{0}\}
$$
then we have the following propositions.

\begin{prop}\label{prop:hyp}
$G_{0}$ is a convergence group.
\end{prop}

\begin{proof}
From \fullref{cor:trans}, we know that the
restriction of $G_{0}$ to each of the components of
$\bb{S}^{1}\setminus\{x_{0},y_{0}\}$ is conjugate to the action of
$\bb{R}$ on itself by translation. Let $g_{n}$ be a sequence of
distinct elements of $G_{0}$ and take a point
$x\in\bb{S}^{1}\setminus\{x_{0},y_{0}\}$. Then the sequence of
points $g_{n}(x)$ will have a convergent subsequence
$g_{n_{k}}(x)$. If this sequence converges to $x_{0}$ or $y_{0}$,
then from \fullref{prop:dir} so will the sequences
$g_{n_{k}}(y)$ for all $y\in\bb{S}^{1}\setminus\{y_{0}\}$ or
$\bb{S}^{1}\setminus\{x_{0}\}$ respectively.

Let $I_{x}$ be the component of
$\bb{S}^{1}\setminus\{x_{0},y_{0}\}$ containing $x$. Assume that
the sequence of points $g_{n_{k}}(x)$ converges to a point $x'\in I_{x}$. Now let $y$ be a point in the other component,
$I_{y}$ of $\bb{S}^{1}\setminus\{x_{0},y_{0}\}$, and consider the
sequence of points $g_{n_{k}}(y)$ in $I_{y}$. If it had a
subsequence which converged to $x_{0}$ or $y_{0}$ then the
sequence $g_{n_{k}}(x)$ would have to as well. This is impossible
so $g_{n_{k}}(y)$ must stay within a compact subset of $I_{y}$ and
hence $g_{n_{k}}$ has a subsequence, $g_{n_{k_{l}}}$ for which $g_{n_{k_{l}}}(y)$ converges to some point $y'\in I_{y}$.

By \fullref{cor:trans}
there exist self homeomorphisms of $I_{x}$ and $I_{y}$ to which the
sequence $g_{n_{k_{l}}}$ converges uniformly on $I_{x}$ and $I_{y}$
respectively. Gluing these together at $x_{0}$ and $y_{0}$ gives
us an element of Homeo($\bb{S}^{1}$) which $g_{n_{k}}$
converges to uniformly. Consequently, $G_{0}$ is a convergence
group.
\end{proof}

\begin{prop}\label{prop:para}
$\bar{G}$ is a convergence group.
\end{prop}

\begin{proof}
Let $f_{n}$ be a sequence of elements of $\bar{G}$. If for every
$y\in\bb{S}^{1}\setminus\{x_{0}\}$ every convergent subsequence of
$f_{n}(y)$ converges to $x_{0}$ then we would be done. So assume
that this is not the case, take $y\in\bb{S}^{1}\setminus\{x_{0}\}$
such that the sequence of points $f_{n}(y)$ has a convergent
subsequence $f_{n_{k}}(y)$ converging to some point
$\tilde{y}\neq x_{0}$. Let $I$ be a small open interval around
$\tilde{y}$, not containing $x_{0}$ then since $G$ is continuously 3--transitive, there
exists a map $F_{\tilde{y}}\co I\to \bar{G}$ satisfying the
following,

\begin{enumerate}
\item{$F_{\tilde{y}}(x)(\tilde{y})=x$ for all $x\in I$}
\item{$F_{\tilde{y}}(\tilde{y})$ is the identity.}
\end{enumerate}

Let $g_{1},g_{2}\in \bar{G}$ satisfy $g_{1}(\tilde{y})=y_{0}$ and $g_{2}(y_{0})=y$ consider the
sequence,
$$
h_{k}=g_{1}\circ F_{\tilde{y}}(f_{n_{k}}(y))^{-1}\circ f_{n_{k}}\circ g_{2}
$$
of elements of $\bar{G}$. They all fix $y_{0}$, and since $g_{1}\circ
F_{\tilde{y}}(f_{n_{k}}(y))^{-1}$ converges to $g_{1}$ as $k\to\infty$ we have
the following.

\begin{enumerate}
\item{If $h_{k}$ contains a
subsequence $h_{k_{l}}$ such that there exists a homeomorphism $h$
with,
$$
\lim_{l\to\infty}h_{k_{l}}=h\hspace{.3cm}{\rm and}\hspace{.3cm}\lim_{l\to\infty}(h_{k_{l}})^{-1}=h^{-1}
$$
then so does $f_{n_{k}}$.}

\item{Furthermore, if there exist points $x',y'\in\bb{S}^{1}$ and a
subsequence $h_{k_{l}}$ of $h_{k}$ such that,
$$
\lim_{l\to\infty}h_{k_{l}}=x'\hspace{.3cm}{\rm and}\hspace{.3cm}\lim_{l\to\infty}(h_{k_{l}})^{-1}=y'
$$
uniformly on compact subsets of $\bb{S}^{1}\setminus\{y'\}$ and
$\bb{S}^{1}\setminus\{x'\}$ respectively, then so does $f_{n_{k}}$
($x'$ and $y'$ will be replaced by $g_{1}^{-1}(x')$ and
$g_{1}^{-1}(y')$).}
\end{enumerate}

Now, since $G_{0}$ is a convergence group, one of the above situations must occur. Consequently, $\bar{G}=\{g\in G:g(x_{0})=x_{0}\}$ is
a convergence group.
\end{proof}

\begin{prop}\label{prop:gcon}
If $G$ is a subgroup of \H which is continuously 3--transitive but not continuously 4--transitive then $G$ is a convergence group.
\end{prop}

\begin{proof}
This proof is almost identical to the previous one but we write it out in full for clarity.

Choose $x_{0}\in\bb{S}^{1}$ and let $f_{n}$ be a sequence of elements of $G$. Then since $\bb{S}^{1}$ is compact,
the sequence of points $f_{n}(x_{0})$ will have a convergent
subsequence, $f_{n_{k}}(x_{0})$, converging to some point
$\tilde{x}$. Let $I$ be a small open interval around $\tilde{x}$,
then since $G$ is continuously 3--transitive, there exists a map
$F_{\tilde{x}}\co I\to G$ satisfying the following,

\begin{enumerate}
\item{$F_{\tilde{x}}(x)(\tilde{x})=x$ for all $x\in I$}
\item{$F_{\tilde{x}}(\tilde{x})$ is the identity.}
\end{enumerate}

Let $g\in G$ send $\tilde{x}$ to $x_{0}$ and consider the
sequence,
$$
h_{k}=g\circ F_{\tilde{x}}(f_{n_{k}}(x_{0}))^{-1}\circ f_{n_{k}}
$$
of elements of $G$. They all fix $x_{0}$, and since $g\circ
F_{\tilde{x}}(f_{n_{k}}(x_{0}))^{-1}$ converges to $g$ as $k\to\infty$ we have
the following.

\begin{enumerate}
\item{If $h_{k}$ contains a
subsequence $h_{k_{l}}$ such that there exists a homeomorphism $h$
with,
$$
\lim_{l\to\infty}h_{k_{l}}=h\hspace{.3cm}{\rm and}\hspace{.3cm}\lim_{l\to\infty}(h_{k_{l}})^{-1}=h^{-1}
$$
then so does $f_{n_{k}}$.}

\item{Furthermore, if there exist points $x',y'\in\bb{S}^{1}$ and a
subsequence $h_{k_{l}}$ of $h_{k}$ such that,
$$
\lim_{l\to\infty}h_{k_{l}}=x'\hspace{.3cm}{\rm and}\hspace{.3cm}\lim_{l\to\infty}(h_{k_{l}})^{-1}=y'
$$
uniformly on compact subsets of $\bb{S}^{1}\setminus\{y'\}$ and
$\bb{S}^{1}\setminus\{x'\}$ respectively, then so does $f_{n_{k}}$
($x'$ and $y'$ will be replaced by $g^{-1}(x')$ and
$g^{-1}(y')$).}
\end{enumerate}

Now, since $\bar{G}=\{g\in G:g(x_{0})=x_{0}\}$ is
a convergence group $G$ is too.
\end{proof}

We now look at the case where $G$ is continuously 4--transitive. In
this case, we show that $G$ must be $n$--transitive for every
$n\in\bb{N}$.

\begin{thm}\label{thm:ncts}
If $G$ is continuously $n$--transitive for $n\geq 4$, then it is
continuously $n+1$--transitive.
\end{thm}

\begin{proof}
Fix $n\geq 4$ and assume for contradiction that $G$ is
continuously $n$--transitive but not continuously $n+1$--transitive.
Take $(a_{1},\dots,a_{n-2})\in P_{n-2}$ and define,
$$
\bar{G}=\{g\in G:g(a_{i})=a_{i}\hspace{.2cm}\forall i\}
$$
Let $I$ be a component of
$\bb{S}^{1}\setminus\{a_{1},\dots,a_{n-2}\}$. Construct a homomorphism $\Psi\co \bar{G}\to\mathrm{Homeo}(\bb{S}^{1})$ in the same way as $\Phi\co G_{0}\to\mathrm{Homeo}(\bb{S}^{1})$ was constructed in \fullref{sec:con}. Explicitly, take $g\in\bar{G}$, restrict it to a self homeomorphism of $\bar{I}$ and identify the endpoints to get an element of $\mathrm{Homeo}(\bb{S}^{1})$.

Let $\bar{\mathcal{G}}$ denote the image of $\bar{G}$ under $\Psi$. Then as in \fullref{prop:iso} $\bar{\mathcal{G}}$ is isomorphic to $\bar{G}$. Using the arguments from the earlier Propositions in this section we can show that $\bar{\mathcal{G}}$ is a convergence group and hence conjugate to a subgroup of $\mathrm{PSL}(2,\bb{R})$. On the other hand, $\bar{\mathcal{G}}$ is 2--transitive on $I$ and every element fixes the identification point. This means that the action of $\bar{G}$ on $I$ must be conjugate to the action of $\mathrm{Aff}(\bb{R})$ on $\bb{R}$.

Let $I$ and $I'$ be two components of
$\bb{S}^{1}\setminus\{a_{1},\dots,a_{n-2}\}$ and let $\phi\co I\to
\bb{R}$ be a homeomorphism which conjugates the action of
$\bar{G}$ on $I$ to the action of Aff($\bb{R}$) on $\bb{R}$. Let
$a_{n-1},a_{n-1}'$ be two distinct points in $I'$. Consider the
groups
$$
G_{0}=\{g\in\bar{G}:g(a_{n-1})=a_{n-1}\}
$$
and
$$
G_{0}'=\{g\in\bar{G}:g(a_{n-1}')=a_{n-1}'\}
$$
They each act transitively on $I$ and by \fullref{cor:fix}
and \fullref{prop:equ} without fixed points. Consequently,
$\phi$ conjugates both of these actions to the action of $\bb{R}$
on itself by translation. Let $g\in G_{0}$ and $g'\in G_{0}'$ be
elements which are conjugated to $x\mapsto x+1$ by $\phi$. Then
$g^{-1}\circ g'$ acts on $I$ as the identity. However, if it is
equal to the identity, then $g'=g$ fixes $a_{n-1}$ and $a_{n-1}'$,
this is impossible as non-trivial elements of $\bar{G}$ can have
at most one fixed point in $I'$. So $g^{-1}\circ g$ is a
non-trivial element of $G$ which acts as the identity on $I$ and
so by \fullref{cor:int} we have that $G$ is continuously
$n+1$--transitive.
\end{proof}

\section{Summary of Results}

\begin{thm}\label{thm:cat}
Let $G$ be a transitive subgroup of \H which contains a non
constant continuous path. Then one of the following mutually
exclusive possibilities holds:
\begin{enumerate}
\item{$G$ is conjugate to $\mathrm{SO}(2,\bb{R})$ in \H.}
\item{$G$ is conjugate to $\mathrm{PSL}(2,\bb{R})$ in \H.}
\item{For every $f\in$\H and each finite set of points $x_{1},\dots,x_{n}\in\bb{S}^{1}$ there exists $g\in G$ such that
$g(x_{i})=f(x_{i})$ for each $i$.}
\item{$G$ is a cyclic cover of a conjugate of $\mathrm{PSL}(2,\bb{R})$ in
\H and hence conjugate to $\mathrm{PSL}_{k}(2,\bb{R})$ for some
$k>1$.}
\item{$G$ is a cyclic cover of a group satisfying condition 3 above.}
\end{enumerate}
\end{thm}

\begin{proof}
Let $f\co [0,1]\to G$ be a non constant continuous path. Then
$$
f(0)^{-1}\circ f\co [0,1]\to G
$$
is a continuous deformation of the
identity in $G$. Consequently, \fullref{prop:one} tells us that $G$ is continuously
1--transitive.

If $J_{x}=\emptyset$ for every $x\in\bb{S}^{1}$ then by \fullref{thm:so} $G$ is conjugate to $\mathrm{SO}(2,\bb{R})$ in \H. If
$J_{x}\neq\emptyset$ for some and hence all $x\in\bb{S}^{1}$ then
by \fullref{thm:nempty} $G$ is either continuously
2--transitive or is a cyclic cover of a group $G'$ which is
continuously 2--transitive.

So assume that $G$ is continuously 2--transitive, then by
\fullref{prop:ttt} it is continuously 3--transitive. If
moreover $G$ is not continuously 4--transitive, then by \fullref{prop:gcon} it is a convergence group and hence conjugate to a
subgroup of $\mathrm{PSL}(2,\bb{R})$. On the other hand, since $G$
is continuously 3--transitive, it is 3--transitive, and hence must
be conjugate to the whole of $\mathrm{PSL}(2,\bb{R})$.

If we now assume that $G$ is continuously 4--transitive then by
\fullref{thm:ncts} it is continuously $n$--transitive and hence
$n$--transitive for every $n\in\bb{N}$. So if we take $f\in$\H and
a finite set of points $x_{1},\dots,x_{n}\in\bb{S}^{1}$ there
exists $g\in G$ such that $g(x_{i})=f(x_{i})$ and we are done.
\end{proof}

\begin{thm}\label{thm:clo}
Let $G$ be a closed transitive subgroup of \H which contains a non
constant continuous path. Then one of the following mutually
exclusive possibilities holds:
\begin{enumerate}
\item{$G$ is conjugate to $\mathrm{SO}(2,\bb{R})$ in \H.}
\item{$G$ is conjugate to $\mathrm{PSL}_{k}(2,\bb{R})$ in \H for some $k\geq 1$.}
\item{$G$ is conjugate to $\mathrm{Homeo}_{k}(\bb{S}^{1})$ in \H for some $k\geq 1$.}
\end{enumerate}
\end{thm}

\begin{proof}
Since $G$ is a transitive subgroup of \H which contains a non
constant continuous path, \fullref{thm:cat} applies. It
remains to show that if $G$ satisfies condition 3 in \fullref{thm:cat} then its closure is \H.

To see this, let $f$ be an arbitrary element of
$\mathrm{Homeo}(\bb{S}^{1})$. If we can find a sequence of
elements of $G$ which converges uniformly to $f$ then we shall be
done. So let $\{a_{n}:n\in\bb{N}\}$ be a countable and dense set
of points in $\bb{S}^{1}$. Choose a sequence of maps $g_{n}\in G$
so that $g_{n}(a_{k})=f(a_{k})$ for $1\leq k\leq n$. Then $g_{n}$
will converge uniformly to $f$ so that the closure of $G$ will
equal \H.
\end{proof}

\begin{thm}
$\mathrm{PSL}(2,\bb{R})$ is a maximal closed subgroup of \H.
\end{thm}

\begin{proof}
Let $G$ be a closed subgroup of \H containing
$\mathrm{PSL}(2,\bb{R})$. Then $G$ is 3--transitive and by applying
\fullref{thm:clo} we can see that \H and
$\mathrm{PSL}(2,\bb{R})$ are the only possibilities for $G$.
\end{proof}

\newpage
\bibliographystyle{gtart}
\bibliography{link}

\begin{thebibliography}{}
\providecommand\bibmarginpar{\leavevmode\marginpar}
\def\urlstyle#1{{\tt #1}}

\bibitem{BE}
\textbf{M Bestvina}, \emph{Questions in geometric group theory}
\ Available at \setbox0\hbox{\makeatletter\@url
{http://www.math.utah.edu/~bestvina/}}
\href{http://www.math.utah.edu/~bestvina/}
{\unhbox0}

\bibitem{CJ}
\textbf{A Casson}, \textbf{D Jungreis}, \emph{Convergence groups and {S}eifert
  fibered 3--manifolds}, Invent. Math. 118 (1994) 441--456 \xox{MR}{1296353}

\bibitem{GA}
\textbf{D Gabai},
  \href{http://links.jstor.org/sici?sici=0003-486X(199211)2:136:3%3C447:CGAFG%%
3E2.0.CO%3B2-9} {\emph{Convergence groups are {F}uchsian groups}}, Ann. of
  Math. $(2)$ 136 (1992) 447--510 \xox{MR}{1189862}

\bibitem{GM}
\textbf{F\,W Gehring}, \textbf{G\,J Martin}, \emph{Discrete quasiconformal
  groups I}, Proc. London Math. Soc. $(3)$ 55 (1987) 331--358 \xox{MR}{896224}

\bibitem{GH}
\textbf{{\'E} Ghys}, \emph{Groups acting on the circle}, Enseign. Math. $(2)$
  47 (2001) 329--407 \xox{MR}{1876932}

\bibitem{HI}
\textbf{A Hinkkanen},
  \href{http://links.jstor.org/sici?sici=0002-9947(199003)318:1%3C87:AANCGO%3E%
2.0.CO%3B2--X} {\emph{Abelian and nondiscrete convergence groups on the
  circle}}, Trans. Amer. Math. Soc. 318 (1990) 87--121 \xox{MR}{1000145}

\bibitem{TU}
\textbf{P Tukia}, \emph{Homeomorphic conjugates of {F}uchsian groups}, J. Reine
  Angew. Math. 391 (1988) 1--54 \xox{MR}{961162}

\end{thebibliography}

\end{document}